\documentclass[reqno]{amsart}

\usepackage[utf8]{inputenc}
\usepackage[OT2,T1]{fontenc}
\DeclareSymbolFont{cyrletters}{OT2}{wncyr}{m}{n}
\DeclareMathSymbol{\Sha}{\mathalpha}{cyrletters}{"58}
\usepackage{graphicx}

\usepackage[hyphens,spaces,obeyspaces]{url}
\usepackage[allcolors=blue,hyperindex]{hyperref}
\hypersetup{
           breaklinks=true,   
           colorlinks=true,   
        }

\usepackage{orcidlink}

\usepackage{pgfplots}
\pgfplotsset{compat=1.18}

\usepackage{mathrsfs}
\title{Positive kernels in Lie group representations and Fourier interpolation} 
\date{\today}
\usepackage[foot]{amsaddr}
\usepackage{comment}

\author[Peter Vang Uttenthal]{Peter Vang Uttenthal \orcidlink{0009-0001-0878-8213}}
\address{Department of Mathematics, Aarhus University, Ny Munkegade 118, 1530-421, DK-8000
Aarhus C, Denmark}
\email{petervang@math.au.dk}

\newcommand{\GL}{\operatorname{GL}}
\newcommand{\SL}{\operatorname{SL}}

\newcommand{\Z}{\mathbb{Z}}

\newcommand{\N}{\mathbb{N}}
\newcommand{\R}{\mathbb{R}}
\renewcommand{\C}{\mathbb{C}}

\usepackage{amsthm,amsmath, mathrsfs, mathtools}
\usepackage{amsfonts}
\usepackage{amssymb}
\usepackage{fancyhdr}
\usepackage{IEEEtrantools}
\usepackage{tikz-cd}
\usepackage[english]{babel}
\usepackage[utf8]{inputenc}
\usepackage{csquotes}

\newtheorem{theorem}{Theorem}
\numberwithin{theorem}{section}
\newtheorem{lemma}{Lemma}
\numberwithin{lemma}{section}

\numberwithin{definition}{section}
\newtheorem{proposition}{Proposition}
\numberwithin{proposition}{section}
\newtheorem{corollary}{Corollary}
\numberwithin{corollary}{section}
\newtheorem{remark}{Remark}
\numberwithin{remark}{section}

\numberwithin{conjecture}{section}

\numberwithin{question}{section}

\numberwithin{example}{section}

\begin{document}

\begin{abstract}
For a given function $f: \Gamma \to \C$
on a discrete subspace $\Gamma$ of a noncompact Riemannian symmetric space $X=G/K$,
we construct analytic, noncompactly supported functions $W$ on $G/K$
with $W|_{\Gamma}=f$ that, in addition,
satisfy an invariant differential equation on $G/K$.
The key step is to regard $W$ as an orthogonal projection onto a space of coherent states in a representation of the Lie group $G$
in a Hilbert space that admits a strictly positive definite reproducing kernel.
The techniques are implemented for the heat kernel on $G/K$ and $W$
is expressed in terms of certain closed formulas for linearly independent coherent states.
\end{abstract}

\maketitle
\setcounter{tocdepth}{4}
{
\tableofcontents
}
\section{Introduction}

The classical extension problem of Whitney \cite{Whitney} asks for a continuously differentiable function
$W$ that coincides with the values of some given $f$ on a topological subspace $\Gamma$ of a smooth manifold $X$. 
The approach of the present work is to give new solutions to a class of extension problems of Whitney type by further utilizing the group of symmetries of the ambient of space $X$ of $\Gamma$.
For concreteness,
a list of desirable properties of $W$ is prescribed in advance: not only should $W$ restrict to $f$ on $\Gamma$, but we shall also insist that
$W$ belongs to certain space $\mathscr{L}(X)$ 
of $G$--invariant functions on $X$ 
for a reductive Lie group $G$.
Naturally, the properties of $W$
will be 
shared by the functions in $\mathscr{L}(X)$. We will carry out the details 
in the setting of the heat equation on noncompact Riemannian symmetric spaces. In general, our techniques do not require that $\mathscr{L}(X)$ be a Hilbert space, nor do we require the action of $G$ to be admissible for our arguments to go through. Nevertheless, the results draw mainly upon the representation theory of $G$ and the harmonic analysis on the associated symmetric space $G/K$.

\begin{theorem} Let $G$ be a reductive Lie group in the Harish-Chandra class with a maximal compact subgroup $K$, and suppose 
$X=G/K$ is a noncompact Riemannian symmetric space.
For a given family $(f_t)_{t>0}$ of analytic functions $f_t \in \mathscr{C}^\omega(X)$
and an infinite discrete space $\Gamma \subseteq X$,
there exists a family $(W_t)_{t > 0}$ 
such that 
\begin{enumerate}
\item $W_t \in \mathscr{C}^\omega(X)$ for all $t>0$,\\
\item $W_t(\gamma) = f(\gamma)$ for all $\gamma \in \Gamma$ and all $t>0$,\\
\item $\partial_t W_t(x) = \Delta W_t(x)$ for
all $x\in X$ and all $t>0$.
\end{enumerate}
Fixing $t>0$,
there exists for every $x\in X$
an open neighborhood $U$ of $x$
such that for every $z\in U$, 
\[
W_t(z) =\sum_{\gamma \in \Gamma_U} a_\gamma p_t(z\gamma^{-1})
\]
where $p_t$ is the heat kernel on $G/K$, $\Gamma_U :=\Gamma \cap U$,
and the constants $a_\gamma$ depend only on $U$.
If $G$ is a complex Lie group then in fact 
\[
W_t(z) = \sum_{\gamma \in \Gamma_U}
a_\gamma 
(4\pi t)^{-n/2}e^{-|\rho|^2 t - |H(z\gamma^{-1})|^2/4t} \prod_{\alpha \in \Delta_+} \bigg( \frac{\operatorname{sinh}\langle \alpha, H(z\gamma^{-1}) \rangle }{\langle \alpha, H(z\gamma^{-1}) \rangle}\bigg)^{-1/2}.
\]
\end{theorem}

Next, we will discuss the relation of our work to the atomic decompositions in representation theory, due to Christensen--\'{O}lafsson \cite{atomic} and others, and the Whittaker--Shannon cardinal series expansion.
Let $G$ be a reductive Lie group with a maximal compact subgroup $K$ and consider a Riemannian 
symmetric space $X \simeq G/K$. 
Let $\mathscr{H}^{(\nu)}$ be a Bergman space of holomorphic functions on $X$ for some parameter $\nu$.
The space $\mathscr{H}^{(\nu)}$ admits a reproducing kernel $K: X\times X\to \C$
with the property that 
\begin{equation}
\label{K}
f(z) = \int_X f(x) \overline{K}(x,z) dx
\end{equation}
for all 
$f\in \mathscr{H}^{(\nu)}$. 
For a countable set of points $\Gamma \subset X$, 
the atomic decomposition  
\[
f(z) = \sum_{\gamma \in \Gamma} \lambda_\gamma(f) e_{\gamma}(z) \quad (z\in X).
\]
of a holomorphic function $f$ on $X$ \cite{atomic}
replaces
the kernel $K$ in (\ref{K}) 
by a sum over 
a countable family of 
coherent states $e_z := K(\cdot, z)$ on $X$. In order
to calculate the coefficients $\lambda_\gamma(f)$, however, one must know the function $f$ on the entire bounded symmetric domain $X$.

\noindent In this work, we use techniques from the representation theory of Lie groups on reproducing kernel Hilbert spaces 
to construct
analytic extensions
\begin{equation}\label{W}
W(z) = \sum_\gamma a_\gamma e_{\gamma}(z) \quad (z\in X)
\end{equation}
over countable collections of coherent states $e_{\gamma}$ within any Hilbert space of functions on a symmetric space with a strictly positive definite reproducing kernel. 
By definition, a kernel function $K$ on $X\times X$ is positive definite if the matrix
$(K(z_i,z_j))_{i,j}$ is positive definite for all finite subsets of distinct points in $X$.
The perspective adopted below differs from \cite{atomic} in that
the given function $f$ does not need to be known everywhere on $X$ to compute the coefficients $a_\gamma$ in \eqref{W}. In contrast, we only require the values of the
function $f$ at a finite or a countable collection of points in $X$. 

Our perspective, in this sense, is closer to that in the Whittaker--Shannon expansion from the theory of wavelets \cite[Theorem 2.23, p. 118]{wavelets}, 
which states that any Paley-Wiener function $f$ on $\R$
of a fixed bandwidth $\delta>0$ is
completely determined by its values at the points $\Gamma = \pi \Z$
via the series expansion
\begin{equation} \label{sinus}
f(z) = \sum_{k\in \Z} f(k\pi/\delta) \operatorname{sinc}(\delta z-k\pi).    
\end{equation}
In Section \eqref{PW}, the identity \eqref{sinus} is interpreted in a positive definite reproducing kernel space; however,
our methods work on general Riemannian symmetric spaces $X$
without restrictions on the bandwidth of the functions or the spacing between the points.
For some striking recent generalizations of \eqref{sinus}
to the space of all Schwartz functions, see \cite{Venkatesh} and \cite{Via}.
I am curious if the techniques developed in \cite{Venkatesh} using the cohomology of arithmetic groups $\Gamma$
in $G=\SL_2(\R)$ and its metaplectic covers
can be adapted to mesh well with the point of view of positive definite kernels adopted here.

\section{Acknowledgments}
 I would like to express my gratitude to 
 Bent {\O}rsted and Birgit Speh for sharing their ideas on the subject of this paper. 
The present work is
supported by Villum Fonden (VIL54509).

\section{Positive kernels} \label{sec:pre}
A function $\psi: \R \to \mathbb{C}$ is said to be positive semidefinite if the inequality
\begin{equation} \label{positivesemi}
\sum_{ i,j } c_i\overline{c}_j \psi(x_i-x_j) \geq 0
\end{equation}
holds for all tuples of distinct points $x_1,\ldots, x_d \in \R$ and all 
scalars $c_1,\ldots, c_d \in \mathbb{C}$. 
We will say that $\psi$ is strictly positive definite if the inequality \eqref{positivesemi} is sharp for all choices of $x_j \in X$ and $c_j\in \C$.
The connection between positive definite function theory and representation theory is that \emph{unitary} representations are characterized by their matrix coefficients being positive definite functions; for more details, see for instance \cite[Chapter 3.2]{dinakar}.

\subsection{Orthogonal projections} Let $G$ be a noncompact reductive Lie group in the Harish-Chandra class
with maximal compact subgroup $K$, and suppose
$G/K$ is a Riemannian symmetric space of noncompact type. 
Fix  a unitary irreducible representation $(\pi, \mathscr{H}_\pi)$ 
of $G$ 
on a Hilbert space $\mathscr{H}_\pi$ of functions on $X:=G/K$.
For a fixed $K$--spherical vector $v_0 \in \mathscr{H}_\pi^K$,
the orbit map 
\[
g \longmapsto \pi(g)v_0 
\]
descends to an embedding 
\[
\begin{tikzcd}
\Phi:  G/K \arrow[r, hook] & \mathscr{H}_\pi  
\end{tikzcd}
\]
of the symmemtric space $G/K$ in the generally infinite-dimensional representation space $\mathscr{H}_\pi$.
Suppose $K: X \times X \to \C$ is a strictly positive definite reproducing kernel for
$\mathscr{H}_\pi$.
Let $\Gamma$ be a finite subset of $G/K$, let 
$e_\gamma = K( \cdot,\gamma)$, and 
define the space
\[
V_\Gamma = \operatorname{span}\{e_\gamma \in \mathscr{H}_\pi: \gamma \in \Gamma\}.
\]
For 
$
K_{ij} = \langle e_{\gamma_i}, e_{\gamma_j} \rangle_{\mathscr{H}_\pi}  = K(\gamma_i,\gamma_j),
$
the matrix $(K_{ij})$ is
invertible if $K$ is strictly positive definite, 
and we let $K^{ij} := (K^{-1})_{ij}$.
For any given function $f: \Gamma \to \C$,
the orthogonal projection onto $V_\Gamma$ given by
\[
W(x) = \sum_{\gamma_i \in \Gamma} \bigg( \sum_{\gamma_j \in \Gamma} 
f(\gamma_j) K^{ij}  \bigg)  
e_{\gamma_i}(x) \quad (x\in G/K)
\]
is a member of $\mathscr{H}_\pi$
such that $W(x) = f(x)$ for all $x\in \Gamma$.
The factors $(K^{ij})$ 
can be viewed as corrections for the nonorthogonality of $e_\gamma \in V_\Gamma$.

\subsection{Semigroups} Let $(S, \circ, \ast)$ be a topological semigroup with a continuous involution $(s \mapsto s^\ast)$ 
such that $s^{\ast\ast}=s$ and 
$(s \circ t)^\ast = t^\ast \circ s^\ast$ for all $s,t\in S$.
Then $\varphi: S\to \C$ is said to be a positive definite function 
if $K_\varphi(s,t):=\varphi(s^\ast \circ t)$ is a positive definite kernel on $S\times S$, and we let $\mathscr{P}(S)$
be the cone of of positive definite functions on $S$. 
If $S$ is an abelian group and $\varphi \in \mathscr{P}(S)$, 
and if the involution on $S$ is $g^\ast = -g$, then $\varphi$ 
is said to be positive definite in the group sense;
if $g^\ast = g$ then $\varphi$
is said to be positive definite in the semigroup sense.

Let $\mathscr{P}_{(\geqslant)}$ be the convex cone of positive semidefinite kernels $\varphi: X\times X \to \C$.
Equip $\mathscr{P}_{(\geqslant)}$ with the topology of pointwise convergence, and let
$\mathscr{P}$ be the closed convex subcone 
of strictly positive definite elements of 
$\mathscr{P}_{(\geqslant)}$.

\begin{lemma} \label{Schur}
The cone $\mathscr{P}$
is closed under pointwise multiplication:
If $
\varphi,\varphi'$ are strictly positive definite functions on 
an involution semigroup $X$, 
then the function $\varphi \varphi'$ is strictly positive definite, where
$(\varphi\varphi')(x):= \varphi(x) \varphi'(x)$ for $x\in X$. 
\end{lemma}
In the positive semidefinite case, at least, Lemma \ref{Schur} goes back to  \cite[p. 14]{schur}.
\begin{proof}
Let $\{x_k \in X: 1\leq k \leq m\}$
and set $a_{jk}=\varphi(x_k^\ast x_j)$
and $a'_{jk}=\varphi'(x_k^\ast x_j)$.
Then $A = (a_{jk}) = B B^\ast$
and $A' = (a'_{jk}) = B'(B')^\ast$
for some $B, B' \in \GL_m(\C)$.
Choose functions $\eta_\ell: \{1..m\}\to \C$ for $1 \leq \ell \leq m$
such that
\[
a'_{jk} = \sum_{\ell=1}^m \eta_\ell(j)\bar{\eta}_\ell(k)
\]
for all $1\leq j,k \leq m$.
For $(c_j)_{j} \neq 0$, the strict positivity of $A=(a_{jk})$ implies
\begin{align*}
    \sum_{j,k} c_j \bar{c}_k a_{jk} a'_{jk} &=
    \sum_{j,k} \bigg(c_j \sum_\ell \eta_\ell(j) \bigg) a_{jk} \bigg( \overline{ c_k \sum_\ell \eta_\ell(k)}\bigg) > 0
\end{align*}
unless $ c_j \sum_\ell \eta_\ell(j) = 0$ for all $j$. 
However, the strict positivity of $A'=(a'_{jk})$ implies for the nonvanishing complex vector $(c_j)$ 
that $0 < \sum_{j,k}c_j \bar{c}_k a_{jk}'$,
so from 
\[0 < \sum_{j,k}c_j \bar{c}_k a_{jk}'  = \bigg \vert \sum_j c_j \sum_\ell \eta_\ell(j) \bigg \vert^2
\]
it follows that there is some $j_0 \in \{1..m\}$ for which $c_{j_0} \sum_\ell \eta_\ell(j_0) \neq 0$.
\end{proof}

\begin{lemma}\label{schoenberg}
    A kernel $\psi: X\times X\to \C$ is negative definite if and only if $e^{-t\psi}$ is positive definite for all $t>0$.
\end{lemma}
\begin{proof}
The authors of \cite{semigroups} attribute the proof to \cite{schoen}.
\end{proof}

\subsection{Integral representations} 
Let $(S,+, \ast)$ be an abelian semigroup. 
The dual semigroup $S^\ast$ of $S$ is the set of semicharaters 
$\rho: S \to \C$ with  
$\rho(s+t)=\rho(s)\rho(t)$ for all $s,t \in S$.
Note that every semicharacter is positive definite.
An absolute value on $S$ is a function $\alpha: S\to \R^\times/\Z^\times$ such that $\alpha(e) = 1$, $\alpha(s\circ t) \leq \alpha(s) \alpha(t)$, and $\alpha(s^\ast)=\alpha(s)$ 
for all $s,t \in S$.
A function $f: S\to \C$ is exponentially bounded if there is an absolute value $\alpha$ on $S$ and a constant $C>0$ 
such that 
$f(s) \leq C \alpha(s)$ for all $s\in S$.

The class of positive definite function $f$ on the real line 
are those with an integral representation 
\[
f(x) = \int_{-\infty}^\infty e^{iyx} d \mu(y)
\]
given as a Fourier transform with respect to a compactly supported positive measure $\mu \geq 0$.    
More generally, the lemma below connects positive definite functions with the theory of integral representations. 

\begin{lemma}
If $\mu$ is a compactly supported Radon measure on $S^\ast$, 
then \[
\varphi(s) := \int_{S^\ast} \rho(s) d \mu(\rho) 
\]
is positive definite and exponentially bounded.   Conversely,   if 
$\varphi$ is a positive definite exponentially bounded function on $S$, then $\varphi$ admits the integral representation
    \[
    \varphi(s) = \int_{S^\ast} \rho(s) d\mu(\rho)
    \]
   for a unique compactly supported Radon measure $\mu$ on $S^\ast$. 
\end{lemma}
\begin{proof}
The lemma is a generalization of Bochner's theorem;
the proof of the nontrivial direction can be found in \cite{semigroups}.
\end{proof}

\section{Strict positivity of the heat kernel}\label{heatsection}

To study Whitney extensions on $\R$, we would like to find families of functions that are close to bump functions, but without the property of vanishing everywhere except in the vicinity of a point. 
Such functions arise in the theory of the heat equation. 

The classical heat equation is 
\begin{equation}\label{heatR}
\partial_t u = \Delta_x u, \quad x\in \R,  
\end{equation}
subject to a given boundary condition $f(x) = u(0,x)$. Let
\[
p_t(x) = \frac{1}{\sqrt{4\pi t}} e^{-x^2/(4t)}, \quad t>0
\]
be the Gaussian density of mean zero and standard deviation $\sqrt{2t}$. Then, the convolution 
\begin{equation} \label{integral}
u(t,x) := p_t * f (x) = \int_{-\infty}^\infty f(y) \frac{1}{\sqrt{4\pi t}} e^{-(x-y)^2/4t} dy
\end{equation}
solves \eqref{heatR}.
Letting $\delta$ be the Dirac delta function, one has the distributional convergence   
\[
p_t \to \delta \quad \textrm{as $t\searrow 0$},
\]
hence the boundary condition
$f(x) = \lim_{t \searrow 0}(p_t \ast f)=  f(x)$.
The family $(p_t)_{t>0}$ is a one-parameter semigroup due to the identity  
\[
p_t * p_s = p_{t+s}
\]
for all $t, s>0$, which follows by applying Fourier transforms on both sides, and using 
$\widehat{p}_t(\xi) = e^{-t \xi^2}$ for every $t>0$.
    
Let $f:\Z \to \mathbb{C}$ be a function,
and let $\cup_{\alpha \in A} U_\alpha$ 
be an open cover of $\R$ consisting of bounded open sets.
Fix some $\alpha \in A$,
define $\Gamma_\alpha = \Gamma \cap U_\alpha$ for $\Gamma = \Z$, 
and consider the problem of constructing an analytic function $W = W_\alpha \in \mathscr{C}^\infty(\R)$
with $W(x) = f(x)$ for all $x\in \Gamma_\alpha$
that, additionally, solves the heat equation.
Define
\[
p_t(v,w) := p_t(v-w) = \frac{1}{\sqrt{4\pi t}} e^{-\frac{1}{4t}
(v-w)^2}, \quad t>0.
\]
For $x \in \Gamma_\alpha$,
let $m= \operatorname{card} \Gamma_\alpha$,
and consider the equation
\[
f(x) 
= 
\sum_{n=1}^m a_n\frac{1}{\sqrt{4\pi t}} e^{-\frac{1}{4t}
x^2} e^{-\frac{1}{4t} n^2} e^{\frac{1}{2t}
xn}
\]
for $(a_n) \in \C^m$, which we rearrange as 
\[
f(x) e^{\frac{1}{4t}
x^2} = 
\sum_{n=1}^{m} a_n \frac{e^{-\frac{1}{4t} n^2}}{\sqrt{4\pi t}} 
\left( e^{-\frac{1}{2t}
x}  \right)^n.
\]
As $x$ traverses $\Gamma_\alpha$, 
we obtain a linear system of dimension $m\times m$,
which is invertible with a solution 
$(a_n')_{n=1}^m$
where
\[
a'_n = a_n \frac{e^{-\frac{1}{4t} n^2}}{\sqrt{4\pi t}}, \quad n\leq m.
\]
The coefficients $a_n$ are then easily recovered from $a_n'$.  
The result is the function $W(t, \cdot) \in \mathscr{C}^\infty(\R)$ given by 
\begin{equation}\label{localW}
W(t,x) = \sum_{1 \leq n \leq m}  a_n p_t(x-n)
\end{equation}
such that
$W(t,x) = f(x)$ 
for all $x\in \Gamma_\alpha$ 
and, moreover,  
$\partial_t W = \Delta_x W$.\\

\noindent 
In summary, 
given a discrete subgroup of $\R$, 
a solution to the Whitney extension problem for the restriction $f|_{\Gamma_\alpha}$ to a bounded chart in an atlas is given above. 
Our approach was to look for an extension $W$ approximately of the form 
\begin{equation} \label{approx}
W (t,x)  \approx  \int_{-\infty}^\infty f(y) p_t(x,y) dy
\end{equation}
where 
$
p_t(x,y) := p_t(x-y)
$
is the covariance matrix for the heat semigroup. 
In fact, we defined $W$ as an appropriate Riemann sum for the integral \eqref{approx}
with respect to the partition at the points in the discrete set $\Gamma$.
When calculating the extension $W$ of $f$, 
the essential feature was 
that the matrix
$(p_t(x_j,x_k))_{j,k}$ was invertible for all finite subsets  $x_1,\ldots,x_m$ of $\Gamma$,
which leads to the notion of positive definite functions defined in \eqref{positivesemi}.\\

\noindent 
For the heat equation on the circle $S^1$ given by
\begin{equation}\label{heatS1}
\partial_t = \Delta u, \quad x\in S^1, t>0,
\end{equation}
 the Laplacian $\Delta$ has eigenfunctions  
\[
\Delta_x \varphi_n (x) = -n^2 \varphi_n (x)
\]
with eigenvalue $-n^2$ 
given by $\varphi_n(x) = e^{inx}$
for $x\in \R/2\pi \Z$.
To obtain a solution of \eqref{heatS1} for $x\in S^1$ and $t > 0$, 
we separate the variables by  
forming the product 
$e^{inx} e^{-tn^2}$
of the Laplacian eigenfunction  
and a Gaussian function of the $t$-variable 
scaled by the eigenvalue $-n^2$. By averaging over $\Z$,
the theta series
\[
\theta(t,x) := \frac{1}{2\pi} \sum_{n= -\infty}^{\infty} e^{-t  n^2} e^{ inx}, \quad t>0, 
\]
descends to a function on $\R/\Z$ 
that solves \eqref{heatS1} and
$
\theta(t,x) \to \delta 
$ as distributions for $t \searrow 0$.
Similarly, averaging 
$p_t(x) = \frac{1}{(4\pi t)^{1/2}} e^{-\frac{x^2}{4t}}$ over $\Z$ gives
a function 
\[
\sum_{k = -\infty}^{\infty} \frac{1}{(4\pi t)^{1/2}} e^{-\frac{(x-k)^2}{4t}}
\]
on $\R/\Z$ converging to $\delta$ as $t \searrow 0$. The heat kernel on $S^1$ is known to be unique, 
so we obtain the identity 
\begin{equation} \label{theta}
 \sum_{n= -\infty}^{\infty} e^{-t n^2} e^{2\pi inx}
= 
\sum_{n = -\infty}^{\infty} \frac{1}{(4\pi t)^{1/2}} e^{-\frac{(x-n)^2}{4t}}
\end{equation}
for all $t>0$ and $x\in \R/\Z$.
This is an instance of the Poisson summation formula
\[
\sum_{n \in \Z} f(n) = \sum_{ n\in \Z} \widehat{f}(n)
\]
stating that the average of a sufficiently decaying function $f$  at the points of the discrete group $\Z$ is equal to the average of the Fourier transform of the function over the dual group $\Z^* = \operatorname{Hom}(\Z,\mathbb{C})\simeq \Z$.
We note that M. Vergne 
has proved a generalization of the Poisson summation formula for the dual restricted root space $\mathfrak{a}^\ast$ in the Lie algebra of arbitrary groups \cite{vergneannals, vergnearxiv}.
The identity \eqref{theta} gives a functional equation relating $t$ and $1/t$, reflecting that $\theta$ is an automorphic form 
of weight $1/2$ for the group $G=\SL_2(\R)$ (or, more precisely, its metaplectic double cover). Alternatively,
the Poisson summation formula
\[
    \theta(t, x) =  \sum_{n\in \Z} p_t(x-n)
\] 
expresses the theta series as a nonconstant Whitney extension
of the constant function
\[
f(0) = \sum_{k=-\infty}^\infty e^{-t k^2}.
\]

\subsection{Gluing} \label{sec:patch}
Let $M$ be a manifold and
let $\Gamma \subset M$ be discrete.
Let $y: \Gamma \to \mathbb{C}$ be a function.
Suppose the manifold $M$ has a locally finite atlas.
Then we may choose an open cover
$M = \cup_{\alpha \in A} U_\alpha$ 
such any $x\in M$ meets only finitely many $U_\alpha$. 
Define 
\[
g(x) = \operatorname{card}\{ \alpha \in A: 
x\in U_\alpha \} \in \Z_{\geq 1}.
\]
\begin{lemma}
The function $g: M \to \Z$ is  locally constant, i.e. $g$ is continuous when $\Z$ carries the discrete topology. 
In fact, for each $x\in M$, 
the set 
\[
V_x = \bigcap_{\alpha: x\in U_\alpha}  U_\alpha
\]
is an open neighborhood of $x$ such that 
$g(v) = g(x)$ for all $v \in V_x$.
\end{lemma}

\begin{proof}
Since $M$ is locally finite, the set 
$\{\alpha: x\in U_\alpha \}$ is finite. 
Therefore, the set
 \[
V_x := \bigcap_{\alpha: x\in U_\alpha} U_\alpha 
\] 
is a finite intersection of open sets and hence itself open. 
Likewise, each set in the intersection defining $V_x$ contains $x$, so
$V_x$ is a neighborhood of $x$. 
We will show that $g(v) = g(x)$ for all $v\in V_x$.

Let $v_0 \in V_x$ and 
let $\alpha$ be such that $x \in U_\alpha$.
Then $ v_0 \in V_x =\cap_{\beta: x\in U_\beta} U_\beta \subseteq U_\alpha$.
Hence $\alpha$ has the property that $v_0 \in U_\alpha$.
We conclude
\[
\{ \alpha: x\in U_\alpha \} \subseteq 
\{\alpha:  v_0 \in U_\alpha \}.
\]
As this is true for all $v_0 \in V_x$, 
\[
\{ \alpha: x\in U_\alpha \} \subseteq 
\bigcap_{v_0 \in V_x} \{\alpha:  v_0 \in U_\alpha \}
\subseteq \{ \alpha : V_x \subseteq U_\alpha \}
\subseteq \{ \alpha: x \in U_\alpha\}.
\]

Hence there must be equality of all the sets involved in this sequence of containments, and we conclude for all $v\in V_x$ that 
\[
\{ \alpha: v \in U_\alpha\}
=
\{ \alpha: x \in U_\alpha \}.
\]
In particular, for all $v\in V_x$ it holds that $g(v) = g(x)$.
\end{proof}

\begin{proposition}\label{glue}
Let $\Gamma$ be a discrete subspace of a manifold $M$ with a locally finite atlas $(U_\alpha)_{\alpha \in A}$.
Let
$f: \Gamma \to \mathbb{C}$ be a function,
and assume for all $\alpha \in A$ 
that 
there exists
$W_\alpha \in \mathscr{C}^\infty(M)$ 
such that 
$W_\alpha(x) = f(x)$ for every 
$x\in \Gamma_\alpha = \Gamma \cap U_\alpha$.
Then
\[
W(x) :=\frac{1}{\operatorname{card}\{\alpha \in A: x \in U_\alpha \}} \sum_{\alpha \in A: x\in U_\alpha} W_\alpha (x) 
\]
belongs to
$\mathscr{C}^\infty(M)$
and satisfies $W(x) = f(x)$ for all $x\in \Gamma$.
\end{proposition}
\begin{remark}
In prop. \ref{glue},
noncompact support of $W_\alpha$ 
    for all $\alpha \in A$ is allowed. 
\end{remark}
\begin{proof}[Proof of \ref{glue}]
Let 
$x\in \Gamma$, and note for any $U_\alpha$ containing $x$ 
that since $x\in \Gamma \cap U_\alpha$ and since $W_\alpha \in \mathscr{C}^\infty(M)$ is a local extension of $f|_{\Gamma_\alpha}$, 
one has 
\[
W_\alpha(x) = f(x). 
\]
Consequently, 
\[
W(x) = \frac{1}{\operatorname{card}\{\alpha \in A: x\in U_\alpha\}} \sum_{\alpha \in A: x \in U_\alpha} f(x) =  
 f(x).
\]
The final step is prove that $W \in \mathscr{C}^\infty(M)$:
Let $x\in M$ and consider the open neighborhood 
\[
V_x = \bigcap_{\alpha: x\in U_\alpha} U_\alpha
\]
around $x$, where we know that 
\[
\{ \alpha \in A : v \in U_\alpha \}
= \{ \alpha \in A : x \in U_\alpha \}
\]
for all $v \in V_x$.
Note for all $v\in V_x$ that
\begin{align}    
W(v) &=\frac{1}{\operatorname{card}\{\alpha \in A: x\in U_\alpha\}}
\sum_{ \alpha \in A : x \in U_\alpha  } W_\alpha(v)
\label{eq:global}
\end{align}
Whenever $x\in U_\alpha$ we know for $V_x \subseteq U_\alpha$
that $W_\alpha$ is smooth when restricted to $V_x$.
The sum \eqref{eq:global} expresses $W$ as a finite sum 
of smooth functions on the open neighborhood $V_x$ of $x$.
This proves that $W$ is smooth in an open neighborhood around $x$.
Since $x \in M$ was arbitrary, we conclude that $W \in \mathscr{C}^\infty(M)$.
\end{proof}

\begin{corollary}
Let $(f_t)_{t>0}$
be a family of functions 
 $f_t: \Z \to \mathbb{C}$. 
There exists a family $(W_t)_{t > 0}$ 
where $W_t \in \mathscr{C}^\infty(\R)$
such that $W_t(\gamma) = f(\gamma)$ for all $\gamma \in \Z$
and $\partial_t W_t(x) = \Delta_x W_t(x)$ for
all $x\in \R$ and $t>0$.
If we fix $t>0$,
then for every $x\in \R$,
there is an open neighborhood $U$ of $x$
such that for all $z \in U$, 
\[
W_t(z) =\sum_{1 \leq k \leq m} a_k p_t(z-k)
\]
for $m \geq 1$ 
and constants $a_k$ 
depending only on $U$.
\end{corollary}
\begin{proof}
Apply proposition \ref{glue} 
to the functions given in equation \eqref{localW}. 
\end{proof}

Let $M$ be a compact Riemannian manifold.
Let $\{ \phi_j : j\in \Z_{\geq 1} \}$ be a complete set of orthonormal eigenfunctions
of $\Delta$. If $\lambda_j$ denotes the eigenvalue corresponding to $\phi_j$, 
we have an increasing, unbounded sequence $0 \leq \lambda_1 \leq \lambda_2 \leq \cdots$, 
where the eigenvalues are allowed to occur with repetitions. 
The kernel of the heat equation on $M$ is given by the formula \cite[Section (4)]{Fegan}
\begin{equation}\label{kernelM}
K(x,y,t) = \sum_{j \geq 1} \phi_j(x) \overline{\phi_j}(y) e^{-\lambda_j t}.
\end{equation}
Note that for any finite set $x_1, \ldots, x_m \in M$ and $c_1,\ldots,c_m \in \mathbb{C}$,
\[
\sum_{1 \leq i,k \leq m} c_i \overline{c}_k K(x_i, x_k, t) = \sum_{j} \left( \sum_i c_i \phi_j(x_i) \overline{\sum_k c_k \phi_j(x_k)} \right) e^{-t \lambda_j} \geq 0.
\]
Next, we specialize to the case where $M = G$ is a compact connected Lie group.
We would like to know when the positive semidefinite kernel is strictly positive definite.
Suppose for all $j $ that
\begin{equation} \label{stjerne}
    \sum_{1 \leq i\leq m} c_i \phi_j(x_i) = 0.
\end{equation}
By the Gelfand-Raikov theorem, the points of $G$ are separated by its irreducible unitary representations: For every pair $g, h\in G$ with $g\neq h$, there exists 
an irreducible unitary representation $\rho$ on a Hilbert space such that 
$\rho(g) \neq \rho(h)$.
Hence, the points of $G$ are separated by its matrix coefficients, which form a dense subspace in the algebra $\mathscr{C}(G)$ of continuous functions on $G$. 
In other words, $\mathscr{C}(G)$ is a separating algebra, where $1 \neq x \in G$ implies 
the existence of some
$\psi \in \mathscr{C}(G)$ such that 
\[
\psi(x) \neq \psi(1).
\]
In particular, for any finite set of distinct elements $g_1,\ldots, g_m \in G$ there exists $\psi \in \mathscr{C}(G)$ such that
\[
\psi(g_i) = \begin{cases}
    1, & i = 1, \\
    0, & i > 1.
\end{cases}
\]
Consequently,  we conclude that $c_1=0$ in \eqref{stjerne}. 
Then, by similar reasoning, we deduce that  
$c_2,\ldots, c_m = 0$, and therefore the kernel
$K$ in \eqref{kernelM} is strictly positive definite. \\

\noindent The heat equation has a natural generalization to any connected Lie group $G$. For a basis 
$X_1,\ldots, X_d$ of the Lie algebra $\mathfrak{g}$ of $G$, the universal enveloping algebra $\mathscr{U}(\mathfrak{g})$ contains the element
\begin{equation} \label{laplacian}
X_1^2+\cdots+X_d^2.     
\end{equation}
Let $\pi$ be the regular representation of $G$ on $L^2(G)$, and let $\Delta$ be the closure of the operator $\pi(X_1^2+\cdots+X_d^2)$. 
Then $\Delta$ is a self-adjoint negative operator \cite[p. 593]{EdwardNelson}, and 
the heat equation on $G$ is
\[
\partial_t u(t,g)= \Delta u(t,g) \quad t>0, g\in G.
\]
Recall that a vector $x$ in a Banach space $\mathscr{H}$ is said to be analytic with respect to an operator $A$ on $\mathscr{H}$ if $\sum_n \frac{1}{n!}||A^nx||s^n < \infty$
for all sufficiently small $s>0$.
Define, for $t>0$, the operator
\[
P^t = e^{t \Delta}.
\]
Then $P^tP^s = P^{t+s}$ and 
there exists a convolution semigroup $(p_t)_{t>0}$
 of nonnegative analytic functions $p_t \in \mathscr{C}^\omega(G)$
with unit mass
$\int_G p_t(g) dg = 1$ such that 
for any smooth sufficiently decaying function $f$ on $G$, 
\[
P^t f(g) = p_t * f (g) = \int_G f(h) p_t(g h^{-1}) dh, \quad g\in G.
\]
The semigroup $(p_t)$ is refered to as the heat kernel on $G$ and 
$p_t$ is an analytic vector with respect to the regular representation.
By \cite[Thm 4]{EdwardNelson}, any representation $(\pi, \mathscr{H})$ of a Lie group on a Banach space $\mathscr{H}$ has a dense set of analytic vectors.
The proof relies on the trick of smoothing vectors $x\in \mathscr{H}$ with the heat kernel: For every $t>0$, 
\[
P^t_\pi x := \int_G ( \pi(g)x ) p_t(g)dg
\]
exists in $\mathscr{H}$ and defines an analytic vector for $\pi$ 
with the property that 
$P^t_\pi x \to x $ as $t\searrow 0$.

\subsection{Riemannian symmetric spaces of noncompact type}
If $G$ is a Lie group in the \emph{Harish-Chandra class},
it comes with a tuple 
$(G,K, \theta, B)$ with the following properties: 
$G$ is reductive, $K$ is a maximal compact subgroup,
$\theta$ is a Cartan involution on the Lie algebra $\mathfrak{g}$, and $B$ is the Killing form on
$\mathfrak{g}\times \mathfrak{g}$ \cite[Ch. 6,7]{knappbeyond}.
For details on the background needed for this section, we refer to \cite{anker}.
Let $(G,K,\theta, B)$ be a reductive group in the Harish-Chandra class. The Cartan decomposition 
\begin{equation} \label{Cartan}
    \mathfrak{g} = \mathfrak{k} \oplus \mathfrak{p}
\end{equation}
is the eigenspace decomposition 
with respect to the eigenvalues $\{\pm 1\}$ of the Cartan involution $\theta$,
where the $1$--eigenspace $\mathfrak{k}$ is the Lie algebra of $K$
and the $(-1)$--eigenspace is $\mathfrak{p}$.
The Killing form $B$ on $\mathfrak{g}\times \mathfrak{g}$ defines a nondegenerate bilinear form
\begin{equation} \label{bilin}
    \langle X, Y \rangle = -B(X,\theta Y).
\end{equation}
For a maximal abelian subspace $\mathfrak{a}$ of $\mathfrak{p}$ with centralizer 
$\mathfrak{m} = Z_\mathfrak{k}(\mathfrak{a})$ in the compact Lie subalgebra $\mathfrak{k}$, the associated root space decomposition is
\begin{equation}
    \mathfrak{g} = \mathfrak{m} \oplus \mathfrak{a}   
    \oplus_{\alpha \in \Delta} \mathfrak{g}_\alpha.
\end{equation}

Choose a positive Weyl chamber $\mathfrak{a}_+$
corresponding to a system  
$\Delta_+$ of positive roots and a  
nilpotent subalgebra $\mathfrak{n}=\bigoplus_{\alpha \in \Delta_+}\mathfrak{g}_\alpha$.
The multiplicities $m_\alpha = \dim_\R \mathfrak{g}_\alpha$ 
of $\alpha \in \Delta_+$ belong to $\{0,1,2\}$,
and the symmetric space $G/K$ has dimension 
\begin{equation}
    n = \dim_\R \mathfrak{a}+ \sum_{\alpha\in \Delta_+} m_\alpha.
\end{equation}
The right regular representation $\pi$ of $G$
descends to an action of
$\mathfrak{g}$ on 
$\mathscr{C}^\infty(G)$ via differential operators
\begin{equation} \label{Liealgact}
d\pi(Y) f(g) = \lim_{t\searrow 0} \partial_t f(g e^{tY})
\end{equation}
for $Y\in \mathfrak{g}$ and $f\in \mathscr{C}^\infty(X)$. In turn, the action \eqref{Liealgact} extends uniquely to the enveloping algebra $\mathscr{U}(\mathfrak{g})$.
For orthonormal bases 
$\mathscr{B}(\mathfrak{k})$ and $\mathscr{B}(\mathfrak{p})$ with respect to \eqref{bilin},
the Casimir operator
\begin{align*}
    \Omega := 
     \sum_{Z\in \mathscr{B} (\mathfrak{k})}Z^2
     -
     \sum_{Y\in \mathscr{B}(\mathfrak{p})}Y^2
\end{align*}
defines an element, independently of the chosen bases, in the center of $\mathscr{U}(\mathfrak{g})$.
Identifying functions on $X$ with
$\pi(K)$--spherical functions on $G$,
the Laplacian $\Delta$ on $X$ may be defined 
for $f\in \mathscr{C}^\infty(G)$ and $g\in G$ by 
\begin{equation}
    \Delta f (gK)= d\pi(\Omega) f(g).
\end{equation}
For $A = \exp \mathfrak{a}$ in the Iwasawa decomposition $G=NAK$,
the projection $H: G \to \mathfrak{a}$ is defined by
\[
g = n e^{H(g)} k, \quad g\in G.
\]
The heat equation on the symmetric space $G/K$
is a parabolic evolution equation 
given in terms of the Laplacian $\Delta$ as
\[
\partial_t  u = \Delta u, \quad t>0,
\]
subject to a prescribed boundary condition $(x\mapsto u(0,x)) \in \mathscr{C}^\omega(X)$.
For $\lambda \in \mathfrak{a}^*$ and the half--sum 
$\rho := (1/2)\sum_{\alpha \in \Delta_+} m_\alpha \alpha$
of positive roots counted with multiplicities, 
the Laplacian eigenfunctions
\[
\Delta \varphi_\lambda = (|\lambda|^2+|\rho|^2) \varphi_\lambda
\]
are given by 
\[
\varphi_\lambda(x) = \int_K e^{\langle \sqrt{-1} \lambda - \rho, H(xk) \rangle} dk.
\]
For a Schwartz function $f$ on $G/K$,
the Fourier transform on the dual
$\mathfrak{a}^\ast$ was introduced by Harish-Chandra as
\[
\widehat{f}(\lambda) = \int_{G/K} f(x) \bar{\varphi}_\lambda(x)dx.
\]
If $\pi$ is a representation of a semigroup $S$ on a Hilbert space $H$,
then $\rho(s) = \langle \pi(s) \xi,  \xi \rangle_{H}$ is positive definite on $S$.
The Plancherel formula gives an isomorphism
\[
L^2(K\backslash G / K) \simeq L^2( \mathfrak{a} , |\mathbf{c}(\lambda)|^{-2}d\lambda)^{\operatorname{Weyl}_G}
\]
implemented by the Fourier transform $\mathscr{F}f(x) = \int_G f(x) \varphi_\lambda(x) dx$. The spherical functions are given as matrix coefficients 
\[
\varphi_\lambda(x) = \langle \pi(x) 1, 1 \rangle_{L^2(K/M)}\]
for $\lambda \in \mathfrak{a}^\ast$, 
and the expression $\varphi_\lambda(y^{-1}x) = \langle \pi_\lambda(x)1,\pi_\lambda(y)1 \rangle_{L^2(K/M)}$ shows that $\varphi_\lambda$ is positive definite on $X$.

\begin{lemma}
The heat kernel is strictly positive definite on $X$.
For every fixed $t>0$ and $\{x_k \in X: 1 \leq k \leq m\}$, the
functions 
$\{p_t(x_k, \cdot) \in \mathscr{C}^\omega(X) : 1 \leq k \leq m\}$ form a linearly independent family of coherent states on $X$. 
\end{lemma}
\begin{proof}
From the semigroup property $p_t = (p_{t/2})^{ \ast 2}$ 
and Fubini's theorem, 
\[
p_t(x,y) = \int_{X} p_{t/2}(x,g)p_{t/2}(g,y) dg. 
\]
For scalars $\{c_k \in \C: 1\leq k \leq m\}$
and distinct points $\{x_k \in X: 1\leq k \leq m \}$
we obtain as a consequence  
\begin{align*}
    \sum_{ j,k} c_j \Bar{c}_k p_t(x_j,x_k) 
    =\int_X \Bigg| \sum_{k}c_k p_{t/2}(x_k, g) \Bigg|^2 dg.
\end{align*}
Suppose $(g \mapsto \sum_{k}c_k p_{t/2}(x_k, g) )\in \mathscr{C}^\infty(X)$ vanishes identically on $X$. 
If $\delta_{x}$ is the Dirac mesaure concentrated at $x \in X$,
then for every $\phi \in \mathscr{C}_c^{\infty}(X)$,
\[\langle e^{ t \Delta } \delta_x , \phi \rangle = \langle \delta_x,  e^{t\Delta}\phi \rangle = \langle p_{t}(x, \cdot), \phi \rangle,\]
which implies
\[ 0 = \bigg\langle  
\sum_{k} c_k p_{t/2}(x_k, \cdot) , \phi \bigg\rangle 
= \bigg\langle  
e^{(t/2)\Delta} \Big( \sum_{k} c_k \delta_{x_k} \Big)
,\phi \bigg\rangle.
\]
With respect to the Dirac measure $\delta$ at the identity coset $eK \in X$, recall that $\partial_t(p_{t} * \delta) = \Delta(p_t\ast \delta)$ in the sense of 
 distributions and $\lim_{t\searrow 0}(p_t \ast \delta) = \delta$. Passing through Fourier transforms, 
 $\partial_t \widehat{p}_t = \widehat{\Delta}  (\widehat{p}_t)
$
 subject to 
 $\lim_{t\searrow 0}\widehat{p}_t= 1$.
Moreover, if we define
\[
h_t(\lambda) := e^{-(|\rho|^2+|\lambda|)t}, \quad \lambda \in \mathfrak{a}^\ast,
\]
and differentiate in $t$, the result is
$\partial_t h_t = \widehat{\Delta} h_t$ and $\lim_{t\searrow 0} h_t = 1$ constantly on $\mathfrak{a}^\ast$. 
By uniqueness, 
$
\widehat{h}_t = p_t,
$
and it follows that $\sum_k c_k \delta_{x_k} = 0$.
Choosing smooth compactly supported test vectors $\varphi_k \in \mathscr{C}^\infty_c(X)$ such that $\varphi_k(x_k)=1$ and $\varphi_k(x_j)=0$ for $j\neq k$ show that 
$c_k = 0$ for all $k\in \{1..m\}$, as desired.
\end{proof}

\begin{corollary} \label{intrep}
    The heat kernel has the integral representation
    \[
    p_t(x,y) = \int_{\mathfrak{a}^\ast} m_t(\lambda) \varphi_\lambda(h^{-1}g) \frac{d\lambda}{|\mathbf{c}(\lambda)|^2}
    \]
    for $x = gK$ and $y=hK$ in $G/K$
    where $m_t(\lambda) = e^{-t(|\lambda|^2+|\rho|^2)}$.
\end{corollary}

Let $G$ be a complex group. Then for all $H\in \mathfrak{a}$ and $t>0$,
\begin{equation} \label{Gcomplex}
p_t(e^H) = (4\pi t)^{-n/2}e^{-|\rho|^2 t - |H|^2/4t} \prod_{\alpha \in \Delta_+} \bigg( \frac{\operatorname{sinh}\langle \alpha, H \rangle }{\langle \alpha, H \rangle}\bigg)^{-1/2}
\end{equation}
Note that 
$
p_t(e^H) =(4\pi t)^{-n/2}e^{-|\rho|^2 t} \kappa(e^H, e^H)
$
where
\[
\kappa(e^H, e^{H'}) =  e^{- \langle H, H' \rangle/4t}
   \prod_{\alpha > 0} \bigg( \frac{\operatorname{sinh} \sqrt{\langle \alpha, H \rangle}\sqrt{\langle \alpha, H' \rangle}}{\sqrt{\langle \alpha, H \rangle}\sqrt{\langle \alpha, H' \rangle}} \bigg)^{-1/2}
\]

\begin{lemma} The function
    $\kappa$
    is a strictly positive definite kernel on $\mathfrak{a}^+\times \mathfrak{a}^+$
\end{lemma}
\begin{proof}
The kernel $\varphi(x,y) = e^{-|x-y|^p}$
is positive definite on $\R \times \R$ if and only if $p\leq 2$.
By Lemma \ref{Schur} and induction,
we deduce that the function
\begin{align}
 H \mapsto   e^{-|H|^2/4t} = \prod_{k=1}^{\operatorname{rk}_\R G } 
    e^{-H_k^2/4t}
\end{align}
is positive definite on $\bar{\mathfrak{a}}_+$,
where $\operatorname{rk}_\R G = \dim_\R \mathfrak{a}$ is the real rank of $G$.
Next, since the function $\psi(H):= \langle \rho, H \rangle$ is negative definite \cite[Thm. 3.20]{semigroups},
Lemma \ref{schoenberg} implies that $e^{-\psi}$ is positive definite.
It suffices to show the statement for
    \begin{align*}
    \kappa_1(e^H,e^{H'}) &= \frac{\operatorname{sinh} \sqrt{\langle \alpha, H \rangle}\sqrt{\langle \alpha, H' \rangle}}{\sqrt{\langle \alpha, H \rangle}\sqrt{\langle \alpha, H' \rangle}}\\
    &= \sum_{m=0}^\infty  \frac{1}{(2m+1)!} \langle \alpha, H \rangle^m \langle \alpha, H' \rangle^m \\
    &= \Bigg \langle \bigg( \frac{\langle \alpha, H \rangle^m}{\sqrt{(2m+1)!}} \bigg)_{m \in \N_0}, 
     \bigg( \frac{\langle \alpha, H' \rangle^m}{\sqrt{(2m+1)!}} \bigg)_{m \in \N_0}
    \Bigg\rangle_{\ell^2(\mathbb{N}_0)}.
     \end{align*}
Suppose
    \begin{align*}
      0 =   \sum_{i,j} c_i \bar{c}_j \kappa_1(e^{H_i},e^{H_j}) &= \vert\vert 
        ( 1/(2m+1)!) \sum_i c_i  \langle \alpha, H_i \rangle^m )_{m\geq 0} 
        \vert\vert^2_{\ell^2(\N_0)}\\
        &=\sum_{m=0}^\infty |1/(2m+1)!|^2 \left|\sum_i c_i \langle \alpha, H_i \rangle^m \right|^2.
    \end{align*}
Since the coefficients of the power series expansion
$
\frac{
\operatorname{sinh} z}{z} = \sum_{m=0}^\infty \frac{1}{(2m+1)!} z^{2m}
$ are strictly positive we deduce that $\sum_i c_i \langle \alpha, H_i \rangle^m  = 0$ for all $m$. If $\langle \alpha, H_i \rangle$ are distinct for all $i$ then every $c_i$ vanishes. 
The lemma follows, then, by noting that in the decomposition $G=K\exp \bar{ \mathfrak{a}}^+ K$, the $\exp \mathfrak{a}^+$--component of any $g\in G$ is unique.
\end{proof}

\subsection{Spectral cutoffs} \label{PW}

With respect to the decomposition 
$G = K \exp \bar{\mathfrak{a}}_+ K$, there is a general estimate for the heat kernel 
due to Anker and Ostellari \eqref{Anker} of the form 
\begin{align} \label{Anker}
    p_t(e^H) \asymp
    \bigg( \frac{
e^{ -\langle \rho, H \rangle - |H|^2/4t}}{e^{|\rho|^2 t} \sqrt{t}^{n}} \bigg)
    \prod_{\alpha}\Big(1+\langle \alpha, H \rangle \Big) \Big(1+t+\langle \alpha, H \rangle \Big)^{\frac{1}{2}(m_\alpha + m_{2\alpha})-1}
\end{align}
for all $H\in  \bar{\mathfrak{a}}_+$ 
and all $t>0$,
where the product is over the set of all nondivisible roots 
$\Delta_+^0$ in $\Delta_+$. For related formulas and details, we refer to \cite{anker2}. We have kept the assumption that $G$ is complex and the associated formula \eqref{Gcomplex}
because the sign of the kernel function that arises from \eqref{Anker} 
seems more difficult to control.

Next, we connect the heat kernel to the setting of Paley-Wiener spaces, which we verify are naturally reproducing kernel spaces.
Let $\Omega \subseteq \R^n$ be a bounded set. 
The orthogonal projection 
\[
P_\Omega: L^2(\R^n) \to L^2(\Omega)
\]
is multiplication by the indicator function on $\Omega$,
\[
P_\Omega f = 1_\Omega f, \quad f\in L^2(\R^n).
\]
The image under the Fourier transform of the closed subspace $L^2(\Omega)$ is the Paley-Wiener space $
PW_\Omega(\R^n)$, and
if 
$
K:= \widehat{1_\Omega},
$
there is a commutative diagram
\[
 \begin{tikzcd}
 L^2(\R^n) \arrow[r, "\widehat{}"]  \arrow[d, "1_\Omega \cdot"] &  L^2(\R^n) \arrow[d, "K*"] \\
L^2(\Omega) \arrow[r, "\widehat{}"] & PW_\Omega
\end{tikzcd}
\]
The idempotence $P^{2}_\Omega = P_\Omega$ of the orthgonol projection implies  for any $f\in PW_\Omega(\R^n)$ that 
\begin{equation}\label{K_conv_f}
K * f = f.
\end{equation}
Letting $K(x,\xi) := K(x-\xi)$ and $e_x (\xi) := K(\xi,x)$,
and noting that $\overline{K(x,\xi)} = K(\xi,x)$,
the identity \eqref{K_conv_f}
can be stated as 
\[
f(x)
=
\int_{\R^n} f(\xi) \overline{K(\xi,x)} d\xi
= \langle f, e_x \rangle_{L^2(\R^n)}.
\]
For example, if $\Omega$ is a ball in $\R^d$, then $K(x) \,  \propto  \, ||x||^{-d/2} J_{d/2}(||x||)$
where $J_{d/2}$ is a Bessel-function. For $d=1$, 
\[
J_{1/2}(x) = \sqrt{\frac{2}{\pi }}  \frac{\sin x}{\sqrt{x}}, \quad K(x) \, \propto \, \frac{\sin x}{x} = 1 - \frac{x^2}{3!} + \frac{x^4}{5!} - \cdots 
\]
In $PW_\Omega(\R)$, note that
    \[
\widehat{1_\Omega\cdot e^{i\lambda}} (x) = \widehat{1_\Omega} * \widehat{e^{i\lambda}}  (x) =
\widehat{1_\Omega} (x-\lambda) = e_\lambda(x) = \operatorname{sinc}(x-\lambda)
\]
For $\lambda \in \pi \Z$, we recover the orthonormal basis with respect to the cardinal sine function from the Whittaker-Shannon expansion.

Let
$G$ be a reductive group in the Harish-Chandra class,
let $\nu$ be the Plancherel measure on the unitary dual $\widehat{G}$,
and for a fixed spectral set $\Omega$, let
\[
PW_\Omega(G/K) = \{ f \in L^2(G/K):  \operatorname{supp} \widehat{f} \subseteq \Omega \}
\]
\[
K(x,y) = \int_{\Omega} \varphi_\lambda(x) \overline{\varphi_\lambda}(y) d\nu(\lambda)
\]
For $\mu = \sum_{\gamma \in \Gamma} a_\gamma \delta_\gamma$,
\begin{align*}    
\widehat{\mu}(\lambda) &= \int_X \overline{\varphi_{\lambda}}(x) d\mu(x)
=\sum_{\gamma \in \Gamma} a_\gamma \overline{\varphi_{\lambda}}(\gamma)\\
\end{align*}
and
\begin{align*}
\mathscr{F}^{-1}\left(
\widehat{\mu}\vert_{\Omega}
\right)(x) &= \int_{\Omega} \widehat{\mu}(\lambda) 
\varphi_{\lambda}(x)  d\nu(\lambda) \\
&=\sum_{\gamma \in \Gamma} a_\gamma \int_{\Omega}\varphi_\lambda(x) \overline{\varphi_\lambda}(\gamma) d\nu(\lambda)\\
&= \sum_{\gamma \in \Gamma} a_\gamma K(x, \gamma)
\end{align*}

For arbitrary complex scalars $c_1,\ldots,c_d$, we have with respect to the measure
$\mu_c := \sum_{\gamma_i \in \Gamma} c_i \delta_{\gamma_i}$ that
\begin{align*}
    \sum_{i,j} c_i \overline{c_j} K(\gamma_i,\gamma_j)
    &= \int_\Omega \bigg\vert \sum_{\gamma_i} c_i \overline{\varphi_\lambda}(\gamma_i)\bigg \vert^2 d\nu(\lambda) \\
    &= \int_\Omega \Big \vert \widehat{\mu_c }(\lambda) \Big \vert^2 d\nu(\lambda).
\end{align*}
from which strict positivity can be deduced. 
We conclude that there is a correspondence 
\[
\mu_\Gamma = \sum_{\gamma\in \Gamma} a_\gamma \delta_\gamma \longmapsto \mathscr{F}^{-1} \left( \widehat{\mu_{\Gamma}}\big|_{\Omega} \right) = \sum_{\gamma \in \Gamma} a_\gamma K(\cdot , \gamma) \in PW_{\Omega}(X)
\]
between Dirac measures $\mu_\Gamma$ and linear combinations of linearly independent coherent states in
$PW_\Omega(X)$.

For $\lambda\in \mathfrak{a}^\ast$, $k \in K$ and $x\in X=G/K$, denote
\[
\phi_{\lambda, k} (x) = e^{\langle \sqrt{-1}\lambda - \rho, H(xk) \rangle}
\]
so that 
\[
\varphi_\lambda(x) =\int_{K} \phi_{\lambda,k}(x) dk.
\]
Fix a bounded set $\Omega \subset \mathfrak{a}^\ast$ that is invariant under the Weyl group. 
Then for $x = gK$ and $y=hK$ in $X=G/K$, 
\begin{align*}    
K(x,y) &= \int_{\Omega} \int_{K/M}  \phi_{\lambda, k} (x)
\overline{\phi_{\lambda, k}}(y) 
dk \frac{d\lambda}{|\mathbf{c}(\lambda)|^2}\\
&=\int_\Omega \varphi_\lambda(h^{-1}g) \frac{d\lambda}{|\mathbf{c}(\lambda)|^2}
\end{align*}

For $\Gamma=\{\gamma_1,\ldots,\gamma_{d}\} \subseteq X$, 
and complex scalars $C =\{c_1,\ldots,c_d\}$ note that 
\begin{align*}
    \sum_{i,j} c_i \overline{c_j} K(\gamma_i,\gamma_j) &=\int_\Omega \int_{K/M} 
\vert \widehat{\mu}(\lambda,k) \vert^2
    dk \frac{d\lambda}{|\mathbf{c}(\lambda)|^2}
\end{align*}
where 
\begin{align*}
    \widehat{\mu}(\lambda,k) = \sum_{j} c_j \overline{ \phi_{\lambda, k}}(\gamma_j).
\end{align*}

Fix $t>0$, let 
\[
m_t(\lambda) := e^{-(|\lambda|^2+|\rho|^2)t},
\]
and suppose we change the spectral cutoff function from $1_\Omega$ 
to $m_t(\lambda)$.
For every $x = g_x K$ and $y=g_yK$ in $G/K$, the Paley-Wiener kernel
\[
K(x,y) = \int 1_\Omega \varphi_\lambda(g_y^{-1}g_x) d\nu(\lambda)
\]
changes into 
\begin{align*}
    K_t(x,y) = \int_{\mathfrak{a}^\ast} m_t(\lambda)  \varphi_{\lambda}(g_y^{-1}g_x) 
   d\nu(\lambda)
\end{align*}
which we recognize as $p_t(x,y)$ by Corollary \ref{intrep}.

\subsection{The main result}
    
    \begin{theorem} Let $G$ be a reductive Lie group in the Harish-Chandra class with a maximal compact subgroup $K$, and suppose 
$X=G/K$ is a noncompact Riemannian symmetric space.
For a given family $(f_t)_{t>0}$ of analytic functions $f_t \in \mathscr{C}^\omega(X)$
and an infinite discrete space $\Gamma \subseteq G$,
there exists a family $(W_t)_{t > 0}$ 
such that 
\begin{enumerate}
\item $W_t \in \mathscr{C}^\omega(X)$ for all $t>0$,\\
\item $W_t(\gamma) = f(\gamma)$ for all $\gamma \in \Gamma$ and all $t>0$,\\
\item $\partial_t W_t(x) = \Delta W_t(x)$ for
all $x\in X$ and all $t>0$.
\end{enumerate}
Fixing $t>0$,
there exists for every $x\in X$
an open neighborhood $U$ of $x$
such that, letting $\Gamma_U =\Gamma \cap U$, 
\[
W_t(z) =\sum_{\gamma \in \Gamma_U} a_\gamma p_t(z\gamma^{-1})
\]
for every $z \in U$ 
for constants $a_\gamma$ depending only on $U$.
If $G$ is a complex Lie group then in fact 
\[
W_t(z) = \sum_{\gamma \in \Gamma_U}
a_\gamma 
(4\pi t)^{-n/2}e^{-|\rho|^2 t - |H(z\gamma^{-1})|^2/4t} \prod_{\alpha \in \Delta_+} \bigg( \frac{\operatorname{sinh}\langle \alpha, H(z\gamma^{-1}) \rangle }{\langle \alpha, H(z\gamma^{-1}) \rangle}\bigg)^{-1/2}.
\]
\end{theorem}

\begin{proof}
Apply proposition \ref{glue} 
to the functions constructed in Section \ref{heatsection}. 
\end{proof}

\begin{corollary} Let $G$ be complex.
Fix a F{\o}lner sequence $(K_n)_{n \in \N}$  in $G$ with $G = \cup_{n\in \N} K_n$ \cite{Folner},
and suppose there is some $r_0 > 0$ such that if $r> r_0$ then 
\begin{equation} \label{Effiecond}    
\Gamma \cap B_{r}(x) \neq \varnothing
\end{equation}
for every $x\in X$. For all $t>0$, we have 
for every $z\in G/K$ the 
pointwise convergence
\begin{align*}
f_t(z) = \lim_{n \to \infty} \Bigg( \frac{1}{(2 \sqrt{\pi} \sqrt{ t })^n}
\sum_{\gamma_{k,n} \in \Gamma_n} a_{k,n} 
e^{-|\rho|^2 t- |H(z\gamma_{k,n}^{-1})|^2/4t} \prod_{\alpha \in \Delta_+} \Big( \frac{\operatorname{sinh}\langle \alpha, H(z\gamma_{k,n}^{-1}) \rangle }{\langle \alpha, H(z\gamma_{k,n}^{-1}) \rangle}\Big)^{-1/2} \Bigg).
\end{align*}
\end{corollary}

For a fixed $k$, 
note that each coefficient 
in the sequence $(a_{k,n})$
will generally be different for each $n\geq k$ 
as $n$ varies through the positive integers greater than $k$.
The density condition \eqref{Effiecond}, forces the coherent states through $\gamma \in \Gamma$ to span a dense subspace, cf. \cite[p. 1486]{Effie}. We note that orthogonal projections onto the linear span of coherent states, denoted above as $V_\Gamma$, are useful
for different purposes in the proofs throughout \cite{Effie}.

\bibliographystyle{amsplain} 
\bibliography{references.bib}


\end{document}